\documentclass{amsart}
\usepackage{latexsym}
\usepackage{amsmath}
\usepackage{amssymb}
\usepackage[all,cmtip,2cell]{xy}
\UseTwocells
\usepackage{bracket}
\usepackage{fancybox} 
\usepackage{ragged2e}

\newtheorem{thm}{Theorem}[section]
\newtheorem{prop}[thm]{Proposition}
\newtheorem{cor}[thm]{Corollary}
\newtheorem{lem}[thm]{Lemma}

\theoremstyle{definition}
\newtheorem{defn}[thm]{Definition}

\theoremstyle{remark}
\newtheorem{rem}[thm]{Remark}
\newtheorem{ex}[thm]{Example}

\newcommand{\K}{{\mathbb R}}
\newcommand{\D}{{\mathcal D}}
\newcommand{\R}{{\mathbb R}}
\newcommand{\mapright}[1]{%
 \smash{\mathop{%
  \hbox to 1cm{\rightarrowfill}}\limits_{#1}}}
\newcommand{\maprightd}[2]{%
 \smash{\mathop{%
  \hbox to 1.2cm{\rightarrowfill}}\limits^{#1}\limits_{#2}}}
\newcommand{\mapleft}[1]{%
 \smash{\mathop{%
  \hbox to 1cm{\leftarrowfill}}\limits_{#1}}}
\newcommand{\mapleftu}[1]{%
 \smash{\mathop{%
  \hbox to 0.8cm{\leftarrowfill}}\limits^{#1}}}
\newcommand{\maprightu}[1]{%
 \smash{\mathop{%
  \hbox to 1cm{\rightarrowfill}}\limits^{#1}}}
\newcommand{\maprightud}[2]{%
 \smash{\mathop{%
  \hbox to 1cm{\rightarrowfill}}\limits^{#1}_{#2}}}
\newcommand{\mapleftud}[2]{%
 \smash{\mathop{%
  \hbox to 1cm{\leftarrowfill}}\limits^{#1}_{#2}}}

\newcounter{eqn}[section]

\def\theeqn{\textnormal{(\thesection.\arabic{eqn})}}

\def\eqnlabel#1{%
  \refstepcounter{eqn}%
  \label{#1}%
  \leqno{\theeqn}}
\begin{document}

\title{A comparison between two de Rham complexes in diffeology % {\small --DRAFT--}
}

\footnote[0]{

%This research was partially supported by a Grant-in-Aid for Scientific
%Research (B) 25287008 
%from Japan Society for the Promotion of Science.

Department of Mathematical Sciences, 
Faculty of Science,  
Shinshu University,   
Matsumoto, Nagano 390-8621, Japan   
e-mail:{\tt kuri@math.shinshu-u.ac.jp}
}

\author{Katsuhiko KURIBAYASHI}
\date{}

\begin{abstract} 
There are two de Rham complexes in diffeology. The original one is due to Souriau and the other one is the singular de Rham complex defined by a simplicial differential graded algebra. 
We compare the first de Rham cohomology groups of the two complexes within the \v{C}ech--de Rham spectral sequence by making use of the {\it factor map} which connects the two de Rham complexes.  
As a consequence, it follows that the singular de Rham cohomology algebra of the irrational torus $T_\theta$ is isomorphic to
 the tensor product of the original de Rham cohomology and the exterior algebra generated by a non-trivial flow bundle over $T_\theta$. 
\end{abstract}

\maketitle

\section{Introduction} %This manuscript is a sequel to \cite[Appendix C]{K_2019}. 
%This manuscript is developed from \cite[Appendix C]{K_2019}. 
The de Rham complex introduced by Souriau \cite{So} is very beneficial in the study of diffeology; see \cite[Chapters 6,7,8 and 9]{IZ}. 
In fact, the de Rham calculus is 
applicable to not only diffeological path spaces but also more general mapping spaces. 
It is worth mentioning that the de Rham complex is a variant of the codomain of Chen's iterated integral map \cite{Chen}. 
While the complex is isomorphic to the usual de Rham complex 
if the input diffeological space is a manifold, the de Rham theorem does not hold in general. 

In \cite{K_2019}, we introduced another cochain algebra called the {\it singular de Rham complex} via the context of simplicial sets. 
It is regarded as a variant of the cubic de Rham complex introduced by Iwase and Izumida in \cite{I-I} and a diffeological counterpart of the singular de Rham complex in \cite{B-G, Su, Sw}. 

An advantage of the new complex is that 
the de Rham theorem holds for {\it every} diffeological space. Moreover,  the singular de Rham complex enables us to construct 
the Leray--Serre spectral sequence and the 
Eilenberg--Moore spectral sequence in the diffeological setting; see \cite[Theorems 5.4 and 5.5]{K_2019}. 
Furthermore, there exists a natural morphism $\alpha : \Omega(X) \to A(X)$ of differential graded algebras 
from the original de Rham complex $\Omega(X)$ due to Souriau to the new one $A(X)$ 
such that the integration map from $\Omega(X)$ to the cubic cochain complex of $X$ introduced in \cite[Chapter 6]{IZ} factors through 
$\alpha$ up to chain homotopy. Thus the map $\alpha$ is called the {\it factor map}. It is important to mention that the idea of cubic differential forms on a diffeological space in \cite[Definition 4.1]{I-I} is a starting point for our consideration of diffeological de Rham theory. 

The result \cite[Theorem 2.4]{K_2019} asserts that the factor map is a quasi-isomorphism of cochain algebras 
if $X$ is a manifold, a finite dimensional smooth CW complex or a parametrized stratifold; see \cite{I, I-I} and \cite{Kreck} for a smooth CW complex and 
a stratifold, respectively. 
Moreover, the factor map $\alpha$ induces a monomorphism $H(\alpha) : H^1(\Omega(X)) \to H^1(A(X))$ 
for {\it every} diffeological space $X$; see \cite[Proposition 6.11]{K_2019}. 
We are interested in a geometric interpretation of the difference between the two de Rham cohomology groups. 

The aim of this manuscript is to compare the first de Rham cohomology groups 
for the complexes $A(X)$ and $\Omega(X)$ within the \v{C}ech--de Rham spectral sequence \cite{IZ_2019} by means of 
the factor map $\alpha$; see the paragraph before Theorem \ref{thm:main} for details. 
In particular, it is shown that the first singular de Rham cohomology for the irrational torus $T_\theta$ is isomorphic to the direct sum of the original one and the group of equivalence classes of flow bundles over $T_\theta$ with connection $1$-forms; see Corollary \ref{cor:main}. 
As a consequence, we see that, as an algebra, the singular de Rham cohomology $H^*(A(T_\theta))$ is isomorphic to the tensor product of the original de Rham cohomology and the exterior algebra generated by a flow bundle over $T_\theta$; see Corollary \ref{cor:cor2}. 
%Thus, it seems that the singular de Rham cohomology has {\it K-theoretical} information. 

In the following remark, we compare the irrational torus $T_\theta$ and the two dimensional torus 
${\mathbb T}^2$ from homotopical and homological points of view in diffeology. 

\begin{rem}
There exists a diffeological bundle of the form ${\mathbb R} \to {\mathbb T}^2 \stackrel{p}{\to} T_\theta$ whose fibre is contractible; see \cite[Chapter 8]{IZ}. 
It follows from the smooth homotopy exact sequence of the bundle that 
the projection $p$ induces isomorphisms 
\[
\pi_1^D(T_\theta) \cong \pi_1^D({\mathbb T}^2)\cong \pi_1({\mathbb T}^2) \cong{\mathbb Z}^{\oplus 2} \ \  \text{and} \  \ 
\pi_i^D(T_\theta) \cong \pi_i^D({\mathbb T}^2) \cong \pi_i({\mathbb T}^2)= 0
\] 
for $i\ge 2$. Here $\pi_i^D( \ )$ and $\pi_i( \ )$ denote the smooth homotopy group functor and the usual homotopy functor, respectively; see \cite[Chapter 5]{IZ} and \cite[3.1]{C-W} for the smooth homotopy group. 
However, the two tori are not homotopy equivalent to each other. This follows from the result that 
the original de Rham cohomology is a homotopy invariant for diffeological  spaces. In fact,  the de Rham cohomology groups of $T_\theta$ and ${\mathbb T}^2$  are not isomorphic to each other; see \cite[6.88]{IZ}. We observe that $H^1(\Omega(T_\theta)) \cong {\mathbb R}$; see \cite[Exercise 119]{IZ}. 

On the other hand, the singular de Rham cohomology $H^*(A(T_\theta))$ is isomorphic to $H^*(A({\mathbb T}^2))$ as an algebra; 
see \cite[Remark 2.9]{K_2019}. We stress that a non-trivial flow bundle is in $H^*(A(T_\theta))$ as mentioned above but not in $H^*(A({\mathbb T}^2))$. In fact,  each flow bundle over a manifold is trivial because the fibre ${\mathbb R}$ is contractible and hence the bundle has a smooth section; 
see \cite[6.7 Theorem]{Steenrod}. 
\end{rem}

In a more general setting, the singular de Rham complex connects with the polynomial de Rham complex 
via quasi-isomorphisms; see \cite[Corollary 3.5]{K_2019}. 
Thus one might expect that 
rational (real) homotopy theory for non-simply connected spaces (simplicial sets), for example \cite{B-Z, G-H-T, M}, works well 
in developing the de Rham calculus for diffeological spaces.  We will pursue the topic in future work. 

An outline for the article is as follows. In Section 2, we describe our main theorem, Theorem \ref{thm:main}, and its corollaries for the irrational torus.  
Section 3 is devoted to proving the results. Section 4 deals with the injectivity of the edge map of 
the \v{C}ech--de Rham spectral sequence. 

\section{The main theorem}

We begin by recalling the definition of a diffeological space. 

\begin{defn}%(\cite{So}) 
For a set $X$,  a set  $\D^X$ of functions $U \to X$ for each open set $U$ in ${\mathbb R}^n$ and for each $n \in {\mathbb N}$ 
is a {\it diffeology} of $X$ if the following three conditions hold: 
\begin{enumerate}
\item (Covering) Every constant map $U \to X$ for all open set $U \subset {\mathbb R}^n$ is in $\D^X$;
\item (Compatibility) If $U \to X$ is in $\D^X$, then for any smooth map $V \to U$ from an open set $V  \subset {\mathbb R}^m$, the composite 
$V \to U \to X$ is also in $\D^X$; 
\item
(Locality) If $U = \cup_i U_i$ is an open cover and $U \to X$ is a map such that each restriction $U_i \to X$ is in $\D^X$, 
then the map $U \to X$ is in $\D^X$. 
\end{enumerate}
\end{defn}

A pair $(X, \D^X)$ consisting of a set and a diffeology is called a {\it diffeological space}. We call an element of a diffeology $\D^X$ a {\it plot}. 
Let $(X, \D^X)$  be a diffeological space and $A$ a subset of $X$. The {\it sub-diffeology} $\D^A$ on $A$ is defined by the initial diffeology for the inclusion 
$i : A \to X$; that is, $p\in \D^A$ if and only if $i\circ p \in \D^X$. 

For a manifold $M$, let $\D_M$ be the set of all smooth maps from open subsets of Euclidean spaces to $M$. It is readily seen that $\D_M$ is a diffeology of $M$. We call it the {\it standard diffeology} of $M$. 

\begin{defn}
Let $(X, \D^X)$ and $(Y, \D^Y)$ be diffeological spaces. A map $f: X \to Y$ is {\it smooth} if for any plot $p \in \D^X$, the composite 
$f\circ p$ is in $\D^Y$. 
\end{defn}

The original de Rham complex due to Souriau is recalled. Let $(X, \D^X)$ be a diffeological space. 
For an open subset $U$ of ${\mathbb R}^n$, let $\D^X(U)$ be the set of plots with $U$ as the domain and 
$\Lambda^*(U) = \{h : U \longrightarrow \wedge^*(\oplus_{i=1}^{n} {\mathbb R}dx_i ) \mid h \ \text{is smooth}\}$
the usual de Rham complex of the manifold $U$.  Let $\mathsf{Open}$ denote the category consisting of open subsets of Euclidean spaces and smooth maps between them.  
We can regard $\D^X( \ )$ and $\Lambda^*( \ )$  as functors from $\mathsf{Open}^{\text{op}}$ to $\mathsf{Sets}$ the category of sets. 

A $p$-{\it form} is a natural transformation from $\D^X( \ )$ to $\Lambda^*( \ )$. Then the de Rham complex $\Omega(X)$ is the 
cochain algebra of $p$-forms for $p\geq 0$; that is, $\Omega(X)$ is the direct sum of the modules
\[
\Omega^p(X) := \Set{
\xymatrix@C35pt@R10pt{
\mathsf{Open}^{\text{op}} \rtwocell^{\D^X}_{\Lambda^p}{\hspace*{0.2cm}\omega} & 
\mathsf{Sets} }
| \omega \ \text{is a natural transformation}
}
\]
with the cochain algebra structure defined by that of $\Lambda^*(U)$ pointwise.   

We introduce another de Rham complex for a diffeological space, which is called the singular de Rham complex. 
Let ${\mathbb A}^{n}:=\{(x_0, ..., x_n) \in {\mathbb R}^{n+1} \mid \sum_{i=0}^n x_i = 1 \}$ 
be the affine space equipped with the sub-diffeology of ${\mathbb R}^{n+1}$ 
and $(A_{DR}^*)_\bullet$ the simplicial cochain algebra 
defined by $(A^*_{DR})_n := \Omega^*({\mathbb A}^{n})$ for each $n\geq 0$. Here we regard ${\mathbb R}^{n+1}$ as a diffeological space endowed with the 
standard diffeology. 
For a diffeological space $(X, \D^X)$, let $S^D_\bullet(X)$ denote the simplicial set defined by 
\[
S^D_\bullet(X):= \{ \{ \sigma : {\mathbb A}^n \to X \mid \sigma \ \text{is a $C^\infty$-map} \} \}_{n\geq 0}. 
\]
%We mention that $S^D_\bullet( \text{-})$ gives the {\it smooth singular functor} defined in \cite{C-W}.  
The simplicial set and the simplicial cochain algebra $(A_{DR}^*)_\bullet$ give rise to a cochain algebra  
\[
\mathsf{Sets^{\Delta^{op}}}(S^D_\bullet(X), (A_{DR}^*)_\bullet):= \Set{
\xymatrix@C35pt@R10pt{
\mathsf{\Delta}^{\text{op}} \rtwocell^{S^D_\bullet(X)}_{(A_{DR}^*)_\bullet}{\hspace*{0.2cm}\omega} & 
\mathsf{Sets} }
| \omega \ \text{is a natural transformation}
}
\]
whose cochain algebra structure is defined by that of $(A_{DR}^*)_\bullet$. 
In what follows, we call the complex $A(X):=\mathsf{Sets^{\Delta^{op}}}(S^D_\bullet(X), (A_{DR}^*)_\bullet)$ 
the {\it singular de Rham complex} of $X$; see \cite[Section 2]{K_2019} for fundamental properties of the cochain algebra. 
Observe that 
the complex $A(X)$ is a variant of the cubic de Rham complex in \cite{I-I}. 

We recall the factor map   
$\alpha : \Omega (X) \to A(X)$ defined by $\alpha(\omega)(\sigma) = \sigma^*(\omega)$ which is natural with respect to smooth maps between diffeological spaces; see \cite[Section 3.2]{K_2019}. 
As mentioned in the Introduction, if $X$ is a manifold, then the factor map is a quasi-isomorphism. 

In order to describe our results, we further recall a generating family, a nebula, a gauge monoid and the \v{C}ech--de Rham  spectral sequence introduced by 
Iglesias-Zemmour in \cite{IZ, IZ_2019}. 
 
A subset ${\mathcal G}_X$ of a diffeology of $X$ is a {\it generating family} of the diffeology if for any plot $p : U \to X$ and $r \in U$, there exists an open neighborhood $V$ of $r$ such that the restriction $P|_V$ is a constant map or $P|_V = F\circ Q$ for some $F : W \to X$ in ${\mathcal G}_X$ and some smooth map $Q : V \to W$; see \cite[1.68]{IZ}. 

Let $(X, \D^X)$ be a diffeological space. 
Let ${\mathcal G}_X$ be the generating family of $\D^X$ consisting of all plots whose domains are open balls in Euclidean spaces. 
We assume that ${\mathcal G}_X$ contains the set $C^\infty({\mathbb R}^0, X)$; see \cite[1.76]{IZ}.  
%We may assume that the domain of each plot in ${\mathcal G}$ is a ball in ${\mathbb R}^N$ for some $N$. 
Then we define the {\it nebula} ${\mathcal N}_X$ of $X$ associated with ${\mathcal G}_X$ to be the diffeological space 
\[
{\mathcal N}_X := \coprod_{\varphi \in {\mathcal G}_X} \big(\{\varphi\} \times \text{dom}(\varphi)\big)
\]
endowed with the sum diffeology, where $\text{dom}(\varphi)$ denotes the domain of the plot $\varphi$. 
We may write ${\mathcal N}({\mathcal G}_X)$ for ${\mathcal N}_X$ when expressing the generating family. 
It is readily seen that the evaluation map $ev : {\mathcal N}_X \to X$ defined by $ev(\varphi, r)=\varphi(r)$ is smooth. 
The {\it gauge monoid} $\mathsf{M}_X$ is a submonoid of the monoid of endomorphisms on the nebula ${\mathcal N}_X$ defined by 
\[
\mathsf{M}_X :=\{f \in C^\infty({\mathcal N}_X, {\mathcal N}_X) \mid ev\circ f = ev \ \text{and} \ 
\sharp \ \! \text{Supp} \ \!  f < \infty \},
\] 
where $\text{Supp} \  \!f := \{ \varphi \in {\mathcal G} \mid f|_{\{\varphi \} \times \text{dom}(\varphi)} \neq 1_{\{\varphi \} \times \text{dom}(\varphi)} \}$. 
In what follows, we denote the monoid $\mathsf{M}_X$ by $\mathsf{M}$ if the underlying diffeological space is clear from the context. 

The original de Rham complex $\Omega^*({\mathcal N}_X)$ is a left $\mathsf{M}^{\text{op}}$-module whose actions are defined by $f^*$ induced by endomorphisms $f \in {\mathcal N}_X$.  Moreover, the complex $\Omega({\mathcal N}_X)$ is regarded as a two sided $\mathsf{M}^{\text{op}}$-module for which the right module structure is trivial.
Then we have the Hochschild complex $C^{*,*}=\{C^{p,q}, \delta, d_\Omega\}_{p,q \geq 0}$ with 
\[
C^{p,q} = %C^p(\mathsf{M}, \Omega^q({\mathcal N}))=
\text{Hom}_{\K\mathsf{M}^{\text{op}}\otimes\K\mathsf{M}}
(\K\mathsf{M}^{\text{op}}\otimes (\K\mathsf{M}^{\text{op}})^{\otimes p}\otimes\K\mathsf{M}, \Omega^q({\mathcal N}_X))\cong 
\text{map}(\mathsf{M}^p, \Omega^q({\mathcal N}_X)),
\]
where the horizontal map $\delta$ is the Hochshcild differential and the vertical map $d_\Omega$ is 
induced by the de Rham differential on $\Omega^*({\mathcal N}_X)$; see \cite[Subsection 8]{IZ_2019}. 
The horizontal filtration $F^* = \{F^j\}_{j\geq 0}$ defined by $F^j = \oplus_{q\geq j}C^{*,q}$ of the 
the total complex $\text{Tot} \ C^{*,*}$ gives rise to 
a first quadrant spectral sequence $\{{}_\Omega E_r^{*, *}, d_r\}$ converging to the \v{C}ech cohomology 
$\check{H}(X):=HH^*(\K\mathsf{M}, \text{map}({\mathcal G}, \K))$ with 
\[
E_2^{p,q} \cong H^q( HH^p(\K\mathsf{M}^{\text{op}}, \Omega({\mathcal N}_X)), d_\Omega), 
\]
where $HH^*( \text{-} )$ denotes the Hochschild cohomology; see \cite[Subsections 9 and 16]{IZ_2019}.  Observe that 
the differential $d_r$ is of bidegree $(1-r, r)$. 
This spectral sequence is called the {\it \v{C}ech--de Rham  spectral sequence}; see \cite{IZ_2019}. 

The same construction as that of the spectral sequence above is applicable to the singular de Rham complex $A(X)$. Then replacing the original de Rham complex $\Omega(\text{-})$ with $A( \text{-} )$, we have a spectral sequence 
$\{{}_A E_r^{*, *}, d_r\}$. The Poincar\'e lemma for 
the complex $A(\text{-})$ holds; see \cite[Theorem 2.4]{K_2019}. Then it follows that the target of the spectral sequence for $A(X)$ is also the \v{C}ech cohomology $\check{H}(X)$. 
Thus the naturality of the factor map $\alpha : A(X) \to \Omega(X)$ gives rise to a commutative diagram of isomorphisms 
\[
\xymatrix@C40pt@R12pt{
H^1(\Omega(X))\oplus {}_\Omega E_3^{1,0} \ar[rr]^{\Theta}_{\cong} &  & H^1(A({\mathcal N}_X)^{\mathsf{M}})\oplus {}_A E_3^{1,0} \\
& \check{H}^1(X; \R). \ar[ur]_{\text{edge}_2}^{\cong} \ar[ul]^{\text{edge}_2}_{\cong}
}
\]
In fact, the edge homomorphism $\text{edge}_1 :=ev^* : H^*(\Omega(X)) \to {}_\Omega E_2^{0,*}=H^1(\Omega({\mathcal N}_X)^{\mathsf{M}})$ 
induced by the evaluation map $ev : X \to {\mathcal N}_X$ is an isomorphism; see \cite[6. Proposition]{IZ_2019}. Moreover, the morphism 
$\alpha : \Omega(X) \to A(X)$ of cochain algebras induces a map $H(\text{Tot}(\alpha))$ between the total complexes which define the spectral sequences above. Thus the naturality of the  map $\alpha$ enables us to obtain a commutative diagram 
\[
 \xymatrix@C15pt@R2pt{
H^*(\Omega(X)) \ar[dd]_-{H(\alpha)}\ar[r]^-{ev^*}_-{\cong}& H^*(\Omega({\mathcal N}_X)^{\mathsf{M}})  = {}_\Omega E_2^{0, *} \ar@{->>}[r] \ar[dd]_{f(\alpha)_2}&
   {}_\Omega E_\infty^{0, *} \ar@{>->}[r]  \ar[dd]_{f(\alpha)_\infty} & H^*(\text{Tot} \ C^{*,*})  \ar[dd]_{H^*(\text{Tot}(\alpha))} &\\
&  & & &  \check{H}^*(X). \ar[lu]_-{edge_2}^-{\cong}  \ar[ld]^-{edge_2}_-{\cong} \\
H^*(A(X)) \ar[r]^-{ev^*} &H^*(A({\mathcal N}_X)^{\mathsf{M}}) =   {}_A E_2^{0, *} \ar@{->>}[r] 
 &  {}_AE_\infty^{0, *} \ar@{>->}[r] & H^*(\text{Tot} \ 'C^{*,*}) & 
 }
  \eqnlabel{add-1}
\] 

By degree reasons, we see that the surjective maps  ${}_KE_2^{0,1}\to {}_KE_\infty^{0,1}$ are isomorphisms and ${}_KE_3^{1,0}\cong {}_KE_\infty^{1,0}$ for $K=\Omega$ and $A$. 
Thus the map $H^*(\text{Tot}(\alpha))$ yields the homomorphism $\Theta$ which fits in the triangle. As a consequence, we see that the map 
$\Theta$ is an isomorphism. 
Furthermore, the diagram (2.1) allows us to conclude that the map $H^1(\alpha) : H^1(\Omega(X)) \to H^1(A(X))$ is injective; 
see the paragraph after \cite[Proposition 6.12]{K_2019}. 

In a particular case where a diffeological space $X$ appears as the base space of a diffeological bundle (see \cite[Chapter 8]{IZ}), 
we consider the injectivity of the edge homomorphism 
$\text{edge}_1^i:=(ev^*)^i : H^i(A(X)) \to H^i(A({\mathcal N}_X)^{\mathsf{M}})= {}_AE_2^{0,i}(X)$ 
for $i = 1, 2$ in order to 
relate $H^*(\Omega(X))$ to $H^*(A(X))$ in the \v{C}ech--de Rham spectral sequence with the diagram (2.1). 
We observe that 
the restriction of the map $\Theta$ mentioned above to $H^1(\Omega(X))$ is the composite of the monomorphism 
$H(\alpha) : H^1(\Omega(X)) \to H^1(A(X))$ and the map $\text{edge}_1$. This follows from 
the commutativity of the left square in the diagram (2.1).

%Theorem 2.3
\begin{thm} \label{thm:main} 
Let $X$ be a connected diffeological space which admits a diffeological bundle of the form $F \to M \stackrel{\pi}{\to} X$ in which $M$ is a connected manifold and 
$F$ is connected diffeological space. Then 
{\em (1)} the edge homomorphism $\text{\em edge}_1^1$
is injective, and 
{\em (2)}  the dimension of the kernel of $\text{\em edge}_1^2$ is less than or equal to $\dim H^1(A(F))$. 
%
%{\em (2)} The edge map $\text{\em edge}_1 : H^2(A(X)) \to \check{H}^2(X; \R)$ is injective.  
\end{thm}

\begin{ex}
Let $G$ be a diffeological group (see \cite[Chapter 7]{IZ}) and $H$ a subgroup of $G$ with the sub-diffeology. Then we have a diffeological bundle of the form $H \to G \stackrel{\pi}{\to} G/H$, where $\pi$ is the canonical projection and $G/H$ is endowed with the quotient diffeology. Thus if $G$ is a Lie group and $H$ is a connected subgroup which is not necessarily closed, then the bundle 
$\pi : G \to G/H$ with fibre $H$ satisfies the condition in Theorem \ref{thm:main}. 
Assume further that $H^1(A(H))= 0$. By virtue of Theorem \ref{thm:main}, we see that the map $\text{edge}_1^i$ is injective for $i= 1$ and $2$. 
\end{ex}

Before describing corollaries, we recall results on principal ${\mathbb R}$-bundles (flow bundles) in \cite{IZ_2019}.  
For a diffeological space $X$, we consider a Hochschild cocycle 
$\tau : \mathsf{M} \to \Omega^0({\mathcal N}_X)=C^\infty({\mathcal N}_X, {\mathbb R})$ in 
$\text{Ker} \{\delta : C^{1,0} \to C^{2,0}\}$. Then an $\mathsf{M}$-action $A_\tau$ 
on ${\mathcal N}_X \times {\mathbb R}$ is defined by $A_\tau(b, s) = (A(b), s+\tau(A)(b))$. 
The action gives rise to a principal ${\mathbb R}$-bundle of the form $Y_\tau:={\mathcal N}_X\times_\tau {\mathbb R} \to {\mathcal N}_X/ \mathsf{M}\cong {\mathcal N}_X/ev \ {\cong} \ \! X$ over $X$, where $Y_\tau$ is the quotient space of ${\mathcal N}_X\times {\mathbb R}$ by the $\mathsf{M}$-action; 
see \cite[1.76]{IZ}.  More precisely, the equivalence relation is generated by the binary relation which the $\mathsf{M}$-action $A_\tau$ induces. Observe that the second diffeomorphism is given by the evaluation map 
$ev : {\mathcal N}_X \to X$. 
%%%%

Let $\text{Fl}(X)$ be the abelian group of equivalence classes of flow bundles. The sum is given by the quotient of the direct sum of two flow bundles by 
the anti-diagonal action of ${\mathbb R}$; see \cite[Proposition 2]{IZ_2019}. 
Then the map 
${}_\Omega E_1^{1,0} \to \text{Fl}(X)$ defined by assigning the equivalence class of the flow bundle $Y_\tau \to X$ to $[\tau]$ is an isomorphism. 
Moreover, we see that ${}_\Omega E_2^{1,0} = \text{Ker} \{d_\Omega : {}_\Omega E_1^{1,0} \to {}_\Omega E_1^{1,1}\}$ is isomorphic to 
$\text{Fl}^{\tiny{\bullet}}(X)$ the subgroup of $\text{Fl}(X)$ consisting of all equivalence classes of flow bundles over $X$ with {\it connection $1$-forms}; see  \cite[8.37]{IZ}. 

Thanks to the injectivity of the edge homomorphism in Theorem \ref{thm:main} and a result on flow bundles mentioned above, we have 

\begin{cor} \label{cor:main} Let $T_\theta$ be the irrational torus. Then the map $\Theta$ in Theorem \ref{thm:main} 
gives rise to an isomorphism 
$\Theta : H^1(\Omega(T_\theta)) \oplus \text{\em Fl}^{\tiny{\bullet}}(T_\theta) \stackrel{\cong}{\to} H^1(A(T_\theta))$.
\end{cor}

We recall the diffeomorphism $\psi : {\mathbb R}/({\mathbb Z}+\theta {\mathbb Z}) \to T_\theta$ defined by 
$\psi(t) = (0, e^{2\pi i t})$ in \cite[Exercise 31, 3)]{IZ}. Then there exist isomorphisms 
$\Omega(T_\theta) \cong \Omega ({\mathbb R}/({\mathbb Z}+\theta {\mathbb Z}) ) \cong (\wedge^*({\mathbb R}), d \equiv 0)$ which are induced by $\psi$ 
and the subduction ${\mathbb R} \to {\mathbb R}/({\mathbb Z}+\theta {\mathbb Z})$, respectively; see \cite[Exercise 119]{IZ}.
On the other hand, we see that $H^*(A(T_\theta)) \cong \wedge (t_1, t_2)$ as an algebra, where $\deg t_i = 1$; see the proof of Corollary \ref{cor:main}.
Thus the corollary above yields the following result. 

\begin{cor}\label{cor:cor2}
There exists an isomorphism $H^*(A(T_\theta)) \cong \wedge (\Theta(t), \Theta(\xi))$ 
of algebras, where $t \in H^*(\Omega(T_\theta))\cong \wedge(t)$ is a generator and 
$\xi \in \text{\em Fl}^{\tiny{\bullet}}(T_\theta) \cong {\mathbb R}$ is a flow bundle over $T_\theta$ with a connection $1$-form, 
which is a generator of the group $\text{\em Fl}^{\tiny{\bullet}}(T_\theta)$. 
\end{cor}  

\section{Proofs of Theorem \ref{thm:main} and Corollary \ref{cor:main}}
We begin by considering invariant differential forms on nebulae of  dfiffeological spaces.  
\begin{lem}\label{lem:Subduction}
Let $\pi : Y \to X$ be a subduction and ${\mathcal G}_Y$ a generating family of $Y$. 
Then the map $\pi^* : A({\mathcal N}_X) \to A({\mathcal N}_Y)$ induced by $\pi$ gives rise to 
a map $\pi^* :  A({\mathcal N}_X)^{\mathsf{M}_X} \to A({\mathcal N}_Y)^{\mathsf{M}_Y}$, 
where the nebula ${\mathcal N}_X$ is defined by the generating family
$\pi_*{\mathcal G}_Y:=\{\pi \circ \phi \mid \phi \in {\mathcal G}_Y \}$ induced by ${\mathcal G}_Y$.
\end{lem}

\begin{proof}
For  $\omega \in A^*({\mathcal N}_X)^{\mathsf{M}_X}$ and $\eta \in \mathsf{M}_Y$, we show that 
$\eta \cdot \pi^*(\omega) = \pi^*(\omega)$. Let $\sigma : {\mathbb A}^n \to {\mathcal N}_Y$ be an element in $S^D_n({\mathcal N}_X)$, namely a smooth map 
from ${\mathbb A}^n$. Since ${\mathbb A}^n$ is connected, it follows that the image of $\sigma$ 
is contained in a component $\{\phi\} \times \text{dom}(\phi)$ of ${\mathcal N}_Y$. We define a smooth map 
$\underline{\eta} : {\mathcal N}_X \to {\mathcal N}_X$ by $\underline{\eta}(\pi\circ \phi, u)= (\pi\circ \phi', \eta(u))$ and by the identity maps in other components,
where $\eta (\phi, u) = (\phi', \eta(u))$. Since $\phi(u) = ev(\eta(\phi, u))= ev(\eta(\phi', \eta(u)))= \phi'(\eta(u))$, it follows that 
$ev\circ \underline{\eta} = \underline{\eta}$ and hence $\underline{\eta} \in \mathsf{M}_X$. Observe that 
$\pi\circ \eta\circ \sigma = \underline{\eta} \circ \pi\circ \sigma$. Thus we see that 
$(\eta \cdot \pi^*(\omega))(\sigma) = \pi^*(\omega)(\eta\circ \sigma)= \omega (\pi \circ \eta \sigma) = \omega (\underline{\eta} \circ \pi\circ \sigma) = 
(\underline{\eta}\cdot \omega)(\pi\circ \sigma) = \omega(\pi\circ \sigma) = \pi^*(\omega)(\sigma)$. This completes the proof. 
\end{proof}

Under the assumption in Theorem \ref{thm:main}, we have a commutative diagram 
\[
\xymatrix@C35pt@R18pt{
H^*(A(X)) \ar[r]^-{ev^*=\text{edge}_1} \ar[d]_{H^*(A(\pi))}& H^*(A({\mathcal N}_X)^{\mathsf{M}_X}) \ar[d]^{\pi^*} &\hspace{-1.7cm}= {}_AE_2^{0,*}(X)\\
H^*(A(M)) \ar[r]^-{ev^*} & H^*(A({\mathcal N}_M)^{\mathsf{M}_M}) &\hspace{-1.5cm}=  {}_AE_2^{0,*}(M) \\
H^*(\Omega(M)) \ar[u]^{H(\alpha)}_{\cong} \ar[r]_-{ev^*}^-{\cong} & H^*(\Omega({\mathcal N}_M)^{\mathsf{M}_M}) \ar[u]_{f(\alpha)_2}& \hspace{-1.3cm}= {}_\Omega E_2^{0,*}(M). 
}
 \eqnlabel{add-2}
\]
Since $M$ is a manifold, it follows from \cite[Theorem 2.4]{K_2019} that $H(\alpha)$ is an isomorphism. 
Observe that, in constructing the spectral sequences, 
we use the generating family ${\mathcal G}_M$ of $M$ consisting of all plots whose domains are open balls in Euclidean spaces.

(I) On the map $H^*(A(\pi))$: 
The fibre $F$ of the bundle $F\to M\stackrel{\pi}{\to} X$ is a manifold. Then the bundle is a fibration in the sense of Christensen and Wu; see 
\cite[Proposition 4.24 and Corollary 4.32]{C-W}. Therefore, the result \cite[Theorem 5.6]{K_2019}
 enables us to obtain 
the Leray--Serre spectral sequence $\{{}_{LS}E_r^{*,*}, d_r\}$ 
for the bundle. We consider the edge homomorphism 
$
%\[
\text{edge}^i : H^i(A(X)) \stackrel{\cong}{\to} {}_{LS}E_2^{i,0} \to 
{}_{LS}E_\infty^{i,0} \to H^i(A(M)).  
%\] 
$
Observe that the map $\text{edge}^i$ is nothing but the map $H^i(A(\pi))$. 

(I)-(1): For degree reasons, we see that  ${}_{LS}E_2^{1,0} \cong {}_{LS}E_\infty^{1,0}$ in the definition of the edge map. 
Thus $\text{edge}^1$ is injective and then so is $H^1(A(\pi))$. 

(I)-(2): We have a commutative diagram 
\[
\xymatrix@C15pt@R15pt{
H^2(A(X)) \ar[d]_{H^*(A(\pi))}\ar[r]^-{\cong}&  {}_{LS}E_2^{2,0}\ar[r]^-{\cong} &  
\text{Im} \ d_2^{0,1}\oplus {}_{LS}E_3^{2,0} \ar[d]_{pr_2}\\
H^2(A(M))   & {}_{LS}E_{\infty}^{2,0} \ar[l]_-j &  {}_{LS}E_3^{2,0}, \ar[l]_{\cong}\\
}
\]
where $pr_2$ denoted the projection into the second factor and $j$ is the inclusion of the filtration which appears in the spectral sequence. 
Therefore, it follows that 
$\text{Ker} \ \! H^2(A(\pi)) \cong  \text{Im} \ \! d_2^{0,1}$.  

(II) The injectivity of $f(\alpha)_2$: Recall the commutative diagram (2.1). 
By degree reasons, we see that the elements in ${}_\Omega E_2^{0,1}$ are non-exact. 
Since $M$ is a manifold, it follows from the argument in \cite[Section 20]{IZ_2019} that ${}_\Omega E_2^{1,0}$ is trivial and then each element in 
${}_\Omega E_2^{0,2}$ is also non-exact; that is, all elements in ${}_\Omega E_2^{0,2}$ are not in the image of the differential 
$d_2 : {}_\Omega E_2^{1,0} \to {}_\Omega E_2^{0,2}$. 

This yields that the upper-left hand side surjective map in (2.1) is bijective. It turns out that 
the map $f(\alpha)_2$ is injective for $*= 1, 2$ and then the map $(ev^*)^i : H^i(A(M)) \to  H^i(A({\mathcal N}_M)^{\mathsf{M}_M}) =  {}_AE_2^{0,i}(M)$ is injective for $i = 1, 2$. 

\begin{proof}[Proof of Theorem \ref{thm:main}] Consider the commutative diagram (3.1). The injectivity of the maps described in  (I)-(1) and (II) implies the result (1).  Moreover, by (II), we see that $\text{Ker} \ \!\text{edge}_1^2 \subset \text{Ker} \ \! H^2(A(\pi))$. 
The argument (I)-(2) enables us conclude that $\dim \text{Ker} \ \!\text{edge}_1^2 \leq \dim \text{Ker} \ \! H^2(A(\pi)) = 
\dim \text{Im} \ \! d_2^{0,1} \leq \dim H^1(A(F))$. We have the result  (2).
\end{proof} 

Before proving Corollary \ref{cor:main}, we recall a result on the \v{C}ech cohomology of a diffeological torus. 
Let $T_K$ be a diffeological torus, namely a quotient ${\mathbb R}^n/K$ endowed with the quotient diffeology, 
where  $K$ is a discrete subgroup of ${\mathbb R}^n$. 

\begin{prop}\text{\em (}\cite[Corollary]{IZ_2019}\text{\em )}\label{prop:C-G} One has an isomorphism 
$\check{H}^*(T_K, {\mathbb R}) \cong H^*(K; {\mathbb R})$.
Here $H^*(K; {\mathbb R})$ denotes the ordinary cohomology of $K$.  
\end{prop}

\begin{proof}[Proof of Corollary \ref{cor:main}] Let $T_\theta$ be the irrational torus. 
By definition, $T_\theta$ is the diffeological space  $T^2/S_\theta$ endowed with the quotient diffeology,
where $S_\theta$ is the subgroup $\{ (e^{2\pi i t}, e^{2\pi i \theta t}) \in T^2\mid t \in {\mathbb R}\}$ which is diffeomorphic to ${\mathbb R}$ as a Lie group. 
Then we have a principal ${\mathbb R}$-bundle of the form 
${\mathbb R} \to T^2 \to T_\theta$ which is a diffeological bundle; see \cite[8.11 and 8.15]{IZ}. Therefore, the Leray--Serre spectral sequence 
\cite[Theorem 5.4]{K_2019} for the bundle allows us to conclude that $H^*(A(T_\theta))\cong H^*(A(T^2))\cong H^*(\Omega(T^2))\cong \wedge (t_1, t_2)$, where 
$\text{deg} \ t_i = 1$. In particular, $H^1(A(T_\theta)) \cong {\mathbb R}\oplus{\mathbb R}$.  

Moreover, by virtue of Theorem \ref{thm:main}, we see that 
the map $\text{edge}_1 : H^1(A(T_\theta)) \to {}_AE_2^{0,1}$ is a monomorphism. 
Since $T_\theta$ is isomorphic to a diffeological torus of the form ${\mathbb R}/({\mathbb Z} +\theta {\mathbb Z})$
; see \cite[Exercise 31, 3)]{IZ}, 
it follows from Proposition \ref{prop:C-G} that 
$\check{H}^*(T_\theta, {\mathbb R}) \cong H^*({\mathbb Z} +\theta {\mathbb Z}; {\mathbb R}) \cong 
H^*({\mathbb Z} \oplus{\mathbb Z}; {\mathbb R})$. This yields that  
${}_AE_2^{0, 1}\oplus {}_AE_3^{1,0} \cong \check{H}^1(T_\theta, {\mathbb R})
\cong  {\mathbb R}\oplus{\mathbb R}$. The injectivity of the edge map above implies that ${}_AE_3^{1,0}(T_\theta) =0$ and hence the map $\Theta$ induces an isomorphism 
$H^1(\Omega(T_\theta)) \oplus {}_{\Omega}E_3^{1,0}\stackrel{\cong}{\to} H^1(A(T_\theta))$. It follows from 
\cite[Section 19]{IZ_2019} that ${}_{\Omega}E_2^{1,0}\cong \text{Fl}^{\tiny{\bullet}}(T_\theta)$. Furthermore, 
we have $H^2(\Omega(T_\theta))=0$; see \cite[Exercise 119]{IZ}. 
It turns out that ${}_{\Omega}E_2^{1,0}\cong {}_{\Omega}E_3^{1,0}$. We have the result. 
\end{proof}

\section{From the second singular de Rham cohomology to the \v{C}ech cohomology}

We define the edge homomorphism $\text{edge} : H^i(A(X)) \to \check{H}^i(X)$ by the composite of the maps in the lower sequence in (2.1). 
For degree reasons, we see that each element in ${}_AE_2^{0,1}$ the $E_2$-term of the \v{C}ech--de Rham spectral sequence is non-exact. 
Then, the map $\text{edge} :  H^1(A(X)) \to \check{H}^1(X)$ is injective under the same assumption as in Theorem \ref{thm:main}. 
In order to consider the edge map in degree $2$, 
we generalize Lemma \ref{lem:Subduction} introducing a generating family of a multi-set. 
Let $\pi : Y \to X$ be a subduction and ${\mathcal G}_Y$ a generating family of $Y$. We define 
${\mathcal G}_X^{\text{multi}}$ by the multi-set $\coprod_{\phi \in {\mathcal G}_Y} \{\pi \circ \phi\}$.   

\begin{prop}\label{prop:degree2}
Under the same assumption as in Theorem \ref{thm:main}, if $H^1(A(F)) = 0$, then the edge map $H^2(A(X)) \to \check{H}^2(X)$ is injective, where 
$\check{H}^2(X)$ is the \v{C}ech cohomology associated with ${\mathcal G}_X^{\text{multi}}$. 
\end{prop}

\begin{rem}\label{rem:ANimportantRem}
In the proof of \cite[Proposition in \S5]{IZ_2019}, we need the condition (*) for a generating family ${\mathcal G}_X$ that for any plot $P : U \to X$ and each 
$r \in U$, there exists a plot $q : B \to Y$ in ${\mathcal G}_X$ such that $q = P|_B$. To this end, we have chosen the generating family ${\mathcal G}_Y$ 
consisting of all plots whose domains are open balls in Euclidian spaces. 
Let ${\mathcal G}_X^{\text{multi}}$ be the generating multi-family mentioned above. 
Then ${\mathcal G}_X^{\text{multi}}$ also satisfies the condition (*). 
We observe that the inclusion $\pi_*{\mathcal G}_Y \to {\mathcal G}_X^{\text{multi}}$ induces a diffeomorphism 
${\mathcal N}(\pi_*{\mathcal G}_Y)/ev \stackrel{\cong}{\to} {\mathcal N}({\mathcal G}_X^{\text{multi}})/ev$ between nebulae and hence the evaluation map 
gives rise to a diffeomorphism ${\mathcal N}({\mathcal G}_X^{\text{multi}})/ev \stackrel{\cong}{\to} X$; see \cite[1.76]{IZ}. 
\end{rem}

With the notation in Remark \ref{rem:ANimportantRem}, for a map $\eta$ in the monoid $\mathsf{M}_Y$, we define $\underline{\eta}(\pi \circ \phi, r) =  (\pi \circ \psi, \eta(r))$, where $\eta(\phi, r) = (\psi, \eta(r))$. 
Then we have a morphism $\pi' : \mathsf{M}_Y \to \mathsf{M}_X$ of monoids defined  by 
$\pi'(\eta) = \underline{\eta}$. Moreover, we define
\[
\widetilde{\pi}: C_X^{p,q}:=\text{map}(\mathsf{M}_X^p, K^q({\mathcal N}_X)) \to \text{map}(\mathsf{M}_Y^p, K^q({\mathcal N}_Y)) =:C_Y^{p,q}
\]
for $K = \Omega$ and $A$ 
by $\widetilde{\pi}(\varphi)(\eta_1, .., \eta_p) = \pi^*(\varphi(\underline{\eta_1}, ..., \underline{\eta_p}))$. A straightforward calculation shows that $\widetilde{\pi}$ is compatible with the differentials $d_\Omega$, $d_A$ and the Hochschild differential $\delta$. 
Thus we have 

\begin{prop}\label{prop:morphismsOF_SS}
The map $\widetilde{\pi}$ induces a morphism of spectral sequences 
$\{f(\widetilde{\pi})_r\} : \{ {}_KE_r^{*,*}(X), d_r\} \to  \{ {}_KE_r^{*,*}(Y), d_r\}$ for 
$K = \Omega$ and $A$. 
\end{prop}

We are ready to prove the main result in this section. 

\begin{proof}[Proof of Proposition \ref{prop:degree2}] Suppose that there exists a non-zero element $x$ in the kernel of the map 
$\text{edge} : H^2(A(X)) \to \check{H}^2(X)$. We recall the commutative diagram (3.1). For the map $\pi^*$ in the right-hand side, 
we see that $\pi^*= f(\tilde{\pi})_2^{0,*}$. This follows from the construction the morphism $\{f(\widetilde{\pi})_r\}$ of the spectral sequence 
for the singular de Rham complex in Proposition \ref{prop:morphismsOF_SS}. 

The arguments in (I)-(2) and (II) before the proof of Theorem \ref{thm:main} enable us to deduce that 
$ev^*(x) \in {}_AE_2^{0,2}(X)$ and $f(\widetilde{\pi})_2(ev^*(x)) \in {}_AE_2^{0,2}(M)$ 
are non-zero elements. 
We observe that $H^2(A(\pi))$ is injective because $H^1(A(F)) = 0$ by assumption. 
Since $x$ is in the kernel, it follows that $ev^*(x)$ is a $d_2$-exact element; that is, the element $ev^*(x)$ is in the image of the differential $d_2$ in the $E_2$-term of the spectral sequence. 
The naturality of $f(\widetilde{\pi})_2$ implies that 
$f(\widetilde{\pi})_2(ev^*(x))$ is also $d_2$-exact. 
Then, the commutativity of the diagram (2.1) obtained by replacing $X$ with $M$ implies that 
the non-zero element $(ev^*\circ H(\alpha)^{-1}\circ H^*(A(\pi)))(x)$ in ${}_\Omega E_2^{0,2}(M)$ is $d_2$-exact. 
For degree reasons, we see that $d_2^{1,0}$ is nontrivial and then so is ${}_\Omega E_2^{1,0}$. 

On the other hand, since $M$ is a manifold, it follows that ${}_\Omega E_1^{1,0}(M)\cong \text{FL}(M)=0$. 
In fact, the fibre ${\mathbb R}$ of a flow bundle is contractible and then the bundle admits a smooth global section; see 
\cite[6.7 Theorem]{Steenrod} for a differentiable approximation of a section.  
Thus, we have ${}_\Omega E_2^{1,0} = \text{Ker}\{{}_\Omega d : {}_\Omega E_1^{1,0}(M) \to 
{}_\Omega E_1^{1,1}(M)\} =0$, which is a contradiction. This completes the proof.
\end{proof}

\medskip
\noindent
{\it Acknowledgements}. The author is grateful to Patrick Iglesias-Zemmour for his explanation on the role of the gauge monoid in the construction of 
the \v{C}ech--de Rham spectral sequence. The author also thanks the referee for constructive comments and suggestions on the previous version of this manuscript.  

%\section{Appendix}

%In this section, we recall the notion of a flow bundle with a connection 1-form for the readers. 
%We refer the reader to \cite[8.37]{IZ} for the details. 

\end{document}